\documentclass{amsart}

\newtheorem{theorem}{Theorem}[section]
\newtheorem{lemma}[theorem]{Lemma}

\theoremstyle{definition}

\newtheorem{example}[theorem]{Example}

\newtheorem{corollary}[theorem]{Corollary}
\theoremstyle{remark}
\newtheorem{remark}[theorem]{Remark}

\numberwithin{equation}{section}

\begin{document}

\title{Prescribed curvature tensor in locally conformally flat manifolds}

\author{Romildo Pina}
\address{Instituto de Matem\'atica e Estat\'istica, Universidade Federal de Goi\'as, Goi\^ania, Brasil, 74001-970}
\curraddr{Insituto de Matem\'atica e Estat\'istica, Universidade Federal de Goi\'as, Goi\^ania, Brasil, 74001-970}
\email{romildo@ufg.br}
\thanks{The authors were supported in part by CAPES/PROCAD and FAPEG/GO}

\author{Mauricio Pieterzack}
\address{Instituto de Matem\'atica e Estat\'istica, Universidade Federal de Goi\'as, Goi\^ania, Brasil, 74001-970}
\email{mauriciopieterzack@gmail.com}


\date{August, 2014 and, in revised form, .}



\begin{abstract}
In the euclidean space $(\mathbb{R}^n,g)$, with $n \geq 3$, $g_{ij} = \delta_{ij}$, we consider a (0,4)-tensor $R= T \odot g$ where $T=\sum_i f_i (x) dx_i^2$ is a diagonal (0,2)-tensor. 
We obtain necessary and sufficient conditions for the existence 
of a metric  $\bar{g}$, conformal to $g$, such that ${\bar R}=R$, where ${\bar R}$ is the Riemannian curvature tensor of the metric $\bar g$.
The solution to this problem is given explicitly for the special cases of the tensor $R$, including a case where the metric $\bar g$ is complete on $\mathbb{R}^n$. Similar problems are considered for locally conformally flat manifolds.
As applications of these results we exhibit explicit examples of metrics $\bar g$, conformal to $g$ they are solutions for this problem.

\end{abstract}

\maketitle

\section{Introduction}
An interesting and very studied problem in geometry is the problem of prescribed scalar curvature.
When the scalar curvature is constant and $M$ is a compact manifold this is just the famous
Yamabe problem. For more details about the problem of prescribed scalar curvature see \cite{RefAu} and \cite{RefCC}. 
For noncompact manifolds, several authors have obtained
progress in this problem. For more details, see for example \cite{RefAu},\cite{RefJin}, \cite{RefMaOw}, \cite{RefCNi},
\cite{RefKW}, \cite{RefBK} and their references.

Another interesting problem that is very studied in present days is the prescribed Ricci curvature equation that can be proposed as follow:
\begin{center}
\begin{minipage}{1cm}
(P1)
\end{minipage}\hspace*{0.5cm}\begin{minipage}{10cm}
  Given a symetric (0,2)-tensor $T$, defined on a manifold $M^n$, $n\geq 3$,\linebreak does there exist
a Riemannian metric $g$ such that $Ric\, g=T$?
\end{minipage}
\end{center}

When $T$ is nonsingular, that is, its determinant does not vanish,
a local solution of the Ricci equation always exists, as it was shown by DeTurck in \cite{RefDT1}.
When $T$ is singular, but still has constant rank and satisfies certain
appropriate conditions, then the Ricci equation also admits local
solutions (see \cite{RefDT2}). Rotationally symmetric nonsingular tensors were
considered in \cite{RefCao} and other results can be found in \cite{RefDT2}, \cite{RefDT3}, \cite{RefDT4},  \cite{RefH}, \cite{RefL} and \cite{RefAu}.
Recently, A. Pulemotov in \cite{RefP} have considered the prescribed Ricci curvature problem in a solid torus.

Pina and Tenenblat have obtained results for this problem
considering special classes of tensors T and conformal metrics (see \cite{RefPT1}, \cite{RefPT2} and their references).

Another problem related with the problem (P1) is {\it Prescribed Curvature Tensor problem}, that can be proposed as the following:
\begin{center}
\begin{minipage}{1cm}
(P2)
\end{minipage}\hspace*{0.5cm}\begin{minipage}{10cm}
 Given a (0,4)-tensor $R$, defined on a manifold $M^n$, $n\geq 3$, does there exists
a Riemannian metric $g$ such that $R_{g}=R$, where $R_g$ is the Riemannian curvature 
tensor of the metric $g$?
\end{minipage}
\end{center}
We observe that solving the problem (P2) is equivalent to solving a nonlinear system of
partial differential equations of second order. 

There are very few studies in the math literature that deal with this problem.
In 1986, DeTurck and Yang in \cite{RefDY} have considered the following problem: Let $R$ be a smooth nondegenerate (0,4)-tensor over a 3-manifold $M$. Then for any $x \in M$ there exists a smooth metric 
$g$ such that $R_g = R$ in a neighborhood of $x$, where $R_g$ is the Riemannian curvature tensor 
of the metric $g$.
They also obtain the local existence of metrics such that the Ricci curvature $Ric$ and
the scalar curvature $K$ satisfy $Ric_{g} + \lambda K g = Q$ for given $\lambda$ and $Q$ is a symmetric (0,2)-tensor.
In 1994, Kowalski and Belger in \cite{RefKB} have considered in $V$ be a $n$-dimensional real vector space,
equipped with a scalar product and tensors $K^{(l)}$ over $V$ satisfying all the identities for a Riemannian curvature tensor which satisfies certains inequalities, and the authors have proved that 
in this case there always exists an analytic Riemannian metric on one open ball such that the
 derivatives of its curvature tensor are the tensors $K^{(l)}$.

In an attempt to resolve it, we can consider the {\it Prescribed Curvature Tensor problem}
with change conformal of metric.
More precisely, we can consider the following problem:
\begin{center}
\begin{minipage}{1cm}
(P3)
\end{minipage}\hspace*{0.5cm}\begin{minipage}{10cm}
 Let $(M^n, g)$, $n\geq 3$, a Riemannian manifold. Given a (0,4)-tensor $R$, defined on a manifold $M$, does there exist a Riemannian metric $\bar{g}$ conformal to $g$ such that
$\bar{R}_{\bar g}=R$, where $\bar{R}_{\bar g}$ is the Riemannian curvature tensor of the metric $\bar{g}$?
\end{minipage}
\end{center}

The curvature tensor of the metric $g$ on $M$ can be decomposed in
$$
R_g = W_g + A_g \odot g,
$$
where $R_g$ is the Riemannian curvature tensor, $\odot$ is the Kulkarni-Nomizu product, $A_g$ and $W_g$ are the
Schouten and the Weyl tensors of $g$, respectivelly (see \cite{RefBE}).

Since the Weyl tensor is conformally invariant and especially in the case that $g$ is locally conformally flat $W_g=0$, if $g$ is locally conformally flat the Riemann curvature tensor is determined by the Schouten tensor and your decomposition is
\begin{equation}
 R_g = A_g \odot g
\label{eqtcurvature}.
\end{equation}

In this paper, we consider on $(\mathbb{R}^n, g)$, $n\geq 3$, the euclidean space, a (0,4)-tensor $R=T \odot g$, where $T$  is a diagonal (0,2)-tensor
given by $T= \displaystyle\sum_{i}f_{i}(x) dx^2_i$, with $f_{i}(x)$ are smooth functions.
Our interest is to study the {\it Prescribed Curvature Tensor problem} in this particular situation.

In Theorem \ref{theorem41} we provide necessary and sufficient conditions over the tensor $R$ for the existence of a metric $\bar{g}$ conformal to $g$ such that 
$\bar{R}_{\bar{g}}=R$. We also extend this result to locally conformally flat manifolds in
Theorem \ref{theorem46}. We consider particular cases for the tensor $R$ when the solutions for the
 {\it Prescribed Curvature Tensor problem} are given explicitilly. In Theorem \ref{theorem42} we consider the case in which $R = f(g\odot g)$, for a smooth function $f$ and $g$ is the euclidean metric. Unfortunately, in this case
the metric $\bar{g}$ is not complete. In Theorem \ref{theorem43} we obtain a result of nonexistence.
In Theorem \ref{theorem44} we consider tensors $T$ depending on only one fixed variable and from this result
we exhibit examples of complete metrics on $\mathbb{R}^n$ with prescribed Riemannian curvature tensor.

\section{Preliminaires}

Our principal goal is to study the {\it Prescribed Curvature Tensor problem} in locally conformally flat manifolds.
In this section, we set forth the notation and review the necessary background to state the results about this problem in the next section.

Let $(M^n,g)$, a flat manifold and $R = T \odot g$, where $T$ is a diagonal (0,2)-tensor, a (0,4)-tensor defined on $M$. 
We want to find $\bar{g} =\displaystyle\frac{1}{\varphi^2} g $ such that $\bar{R} = R$, where $\bar{R}$ is the Riemannian curvature tensor of the metric $\bar{g}$. That is, we want to study the problem

 \begin{equation}
\left\{ \begin{array}{l}
\bar{g} =\displaystyle\frac{1}{\varphi^2} g  \\
\bar{R} = R
\end{array}
\right.
\label{eqtensorcurvature1}
\end{equation}

Using the decomposition of the Riemannian curvature tensor $\bar{R}$ in (\ref{eqtcurvature}) we obtain that
$\bar{R}=R$ is equivalent to $A_{\bar{g}} \odot \bar{g} = T \odot g$. 
As $\bar{g} =\displaystyle\frac{1}{\varphi^2} g $ this equation is equivalent to 
\[
A_{\bar{g}} \odot \left(\displaystyle\frac{g}{\varphi^2}\right) = T \odot g.
\]

Thus, $\bar{R}=R$ is equivalent to
\[
A_{\bar{g}} \odot g = \left(\varphi^2 T\right) \odot g.
\]

Since the Nomizu-Kulkarni product is injective (see Lemma 1.113 in \cite{RefBE}), the problem
(\ref{eqtensorcurvature1}) is equivalent to
\begin{equation}
\left\{ \begin{array}{l}
\bar{g} =\displaystyle\frac{1}{\varphi^2} g  \\
A_{\bar{g}} = \varphi^2 T
\end{array}
\right.
\label{eqtensorcurvature21}
\end{equation}

From now we are considering the euclidean space $(\mathbb{R}^n, g)$, $n\geq 3$, with  coordinates $x=(x_1,..,x_n)$ and
$ g_{ij} = \delta_{ij}$. Given a (0,4)-tensor $R=T \odot g$, where $T$ is a diagonal (0,2)-tensor given by $T= \displaystyle\sum_{i}f_{i}(x) dx^2_i$,
with $f_{i}(x)$ smooth functions, we seek necessary and sufficient conditions
for the tensor $R=T \odot g$ to the existence of a metric  $\bar{g}= \displaystyle\frac{1}{\varphi^2}g$
such that ${\bar R}=R$.

The Schouten tensor of $\bar{g}$ is defined by
$$
A_{\bar{g}} = \displaystyle\frac{1}{n-2} \left(Ric_{\bar{g}} - \displaystyle\frac{\bar{K}}{2(n-1)}\bar{g} \right),
$$
where $Ric_{\bar{g}}$ and $\bar{K}$ are the Ricci tensor and the scalar curvature of $\bar{g}$, respectively.

As $\bar{g}$ is conformal to the euclidean metric $g$, the Ricci tensor of $\bar g$ is given by
\begin{equation}
\mbox{Ric}_{\bar{g}} = \displaystyle\frac{1}{\varphi^2}\left\{ (n - 2)\varphi Hess_g \varphi +
\left( \varphi \Delta_g \varphi - (n-1)|\nabla_g\varphi|^2\right)g \right \}
\label{eqric}
\end{equation}
\noindent and the scalar curvature of $\bar {g}$ is given by
\begin{equation}
\bar{K} =(n - 1)\left(2\varphi \Delta_g \varphi - n|\nabla_g\varphi|^2
\right)
\label{eqsca}
\end{equation}
\noindent  where $ \Delta_g $ and  $\nabla_g$ denote the laplacian and the gradient in the
euclidean metric $g$, respectivelly (see \cite{RefBE}).

Using the expressions (\ref{eqric}) and (\ref{eqsca}) the Schouten tensor of $\bar{g}$ can be expressed by
\begin{equation}
A_{\bar{g}} = \displaystyle\frac{Hess_g \varphi}{\varphi}
- \displaystyle\frac{|\nabla_g\varphi|^2}{2\varphi^2} g.
\label{eqschtensor}
\end{equation}

We will denote by $\varphi_{,x_k}$ and
$f_{i,x_k}$ the derivatives of $\varphi$ and $f_{i}$ with respect to $x_k$,
respectively, $\varphi_{,x_ix_j}$ and $f_{i,x_ix_j}$ the second order derivatives of $\varphi$ and $f$ with respect to $x_ix_j$, respectively. Since $g$ is the euclidean metric in ${\mathbb R}^n$, $n \geq 3$,
studying the problem (\ref{eqtensorcurvature21})
when $T = \displaystyle\sum_{i}f_{i}(x) dx^2_i$, $f_{i}(x)$ are smooth functions,
 is equivalent to studying the following system of equations

 \begin{equation}
\left\{
\begin{array}{l}
\displaystyle\frac{\varphi_{,x_ix_i}}{\varphi}-
\displaystyle\frac{|\nabla_g\varphi|^2}{2 \varphi^2}= \varphi^2 f_i, \quad
\forall \quad i: 1,...,n.\\
\varphi_{,x_ix_j} = 0, \quad \forall \quad i \ne j.
\vspace*{.1in}
\end{array}
\right.
\label{eqtensorcurvature3}
\end{equation}

From the second equation of (\ref{eqtensorcurvature3}) it follows that $\varphi$ can be
expressed as a sum of functions, each of which depends only on one of the variables $x_i$, so 
we will write $\displaystyle {\varphi (x) = \sum_{i=1}^n  \varphi_i(x_i)}$.

We will studing the system (\ref{eqtensorcurvature3}) with the additional condition that $3f_i(x) + f_j(x)\neq 0$, for all $x \in \mathbb{R}^n$ and all $i\neq j$.

\section{Main results}

We now  state our main results. We start with a lemma which proves to be very useful in the proofs to follows.

\begin{lemma}\label{lemma41}
 Let $\varphi(x_1,...,x_n)$ be a solution of (\ref{eqtensorcurvature3}). Then 

\begin{equation}
\frac{\varphi_{,x_j}}{\varphi} = - \frac{f_{i,x_j}}{3f_{i}+f_{j}}, \qquad \forall i \ne j
 \label{eqlemma41}
\end{equation}
\noindent and
\begin{equation}
 \displaystyle\frac{f_{k,x_j}}{3f_k+f_j} = \frac{f_{i,x_j}}{3f_i+f_j}.
\label{eqlemma411}
\end{equation}
\noindent for distinct $i,j,k$.

\end{lemma}

\begin{proof}
 From the first equation of (\ref{eqtensorcurvature3}) we obtain
\[
\varphi_{,x_ix_i}= \varphi^3 f_i + \displaystyle\frac{\displaystyle\sum_k (\varphi_{,x_k})^2}{2 \varphi}.
\]

Taking the derivative with respect to $x_j, j \neq i$ and using again the first equation of
(\ref{eqtensorcurvature3}) we obtain
\[
\begin{array}{lcl}
    0 & = & 3\varphi^2 \varphi_{,x_j} f_i + \varphi^3 f_{i,x_j} +
\displaystyle\frac{2 \varphi_{,x_j} \varphi_{,x_jx_j}
2\varphi - 2 \varphi_{,x_j} |\nabla_g\varphi|^2}{4\varphi^2}\\
     & = & 3\varphi^2\varphi_{,x_j} f_i + \varphi^3 f_{i,x_j} + \varphi_{,x_j} \displaystyle\frac{\varphi_{,x_jx_j}}{\varphi} -
\displaystyle\frac{ \varphi_{,x_j} |\nabla_g\varphi|^2}{ 2\varphi^2} \\
     & = &  3 \varphi^2 \varphi_{,x_j} f_i + \varphi^3 f_{i,x_j} + \varphi_{,x_j} \varphi^2 f_j +
\displaystyle\frac{\varphi_{,x_j} |\nabla_g\varphi|^2}{2\varphi^2} -
\displaystyle\frac{ \varphi_{,x_j} |\nabla_g\varphi|^2}{2\varphi^2} \\
     & = & \varphi^2 (\varphi_{,x_j} (3f_i+f_j) + \varphi f_{i,x_j})
  \end{array}
\]
Therefore, since $\varphi \neq 0$ we obtain
\begin{equation}
 \varphi_{,x_j} (3f_i+f_j) + \varphi f_{i,x_j} = 0
\label{eqlemma0}
\end{equation}
 and consequently (\ref{eqlemma41}). The equation (\ref{eqlemma411}) follows imediatelly from (\ref{eqlemma41}).
\end{proof}

The relationship between the conformal factor $\varphi$ and the functions $f_i$ that compose the tensor $T$
in this Lemma is fundamental to establish the results of this work.

The next theorem is our main result and contains necessary and suficient conditions for problem
(\ref{eqtensorcurvature1}) to have a solution.

\begin{theorem} \label{theorem41}
Let $(\mathbb{R}^n, g)$, $n\geq 3$, be the euclidean space, with coordinates
  $x_1,...,x_n$, and metric $g_{ij}=\delta_{ij}$. Consider a (0,4)-tensor $R=T\odot g$, where 
$T = \displaystyle\sum_{i=1}^n f_{i}(x)dx_i^{2} $, $f_i(x)$ are smooth functions such that
$3f_i(x) +f_j(x)\neq 0$ for all  $x \in \mathbb{R}^n$ and all $i\neq j$. Then there exists a positive
function $\varphi$ such that the metric $\bar{g} = \displaystyle\frac{1}{\varphi^2} g$ satisfies   ${\bar{R}} = R$ if and only if
the functions $f_i$ satisfy the following system of
differential equations
\begin{equation}
\left\{ \begin{array}{l}
\displaystyle\frac{f_{i,x_j}}{3f_i +f_j} =\displaystyle\frac{f_{k,x_j}}{3f_k + f_j},  \quad i \neq j,\ \ k\neq j, \\
\vspace*{.1in}
\left(\displaystyle\frac{f_{j,x_i}}{3f_j +f_i}\right)_{,x_k} = \left(\displaystyle\frac{f_{j,x_k}}{3f_j +f_k}\right)_{,x_i}, \quad i \neq j,\ \ k\neq j,
 \\
\vspace*{.1in}
\left( \dfrac{f_{i,x_j}}{3f_i +f_j}\right )_{,x_i} = \left(\dfrac{f_{j,x_i}}{3f_j +f_i}\right)_{,x_j}, \quad i \neq j,\\
\vspace*{.1in}
\dfrac{1}{2}\left(\dfrac{f_{j,x_i}}{3f_j +f_i}\right)^2 -
\left(\dfrac{f_{j,x_i}}{3f_j +f_i}\right)_{,x_i}
-  \dfrac{1}{2}\displaystyle\sum_{k \neq i}\left(\frac{f_{j,x_k}}{3f_j +f_k}\right)^2 = h_i, \quad i \neq j, \\
\vspace*{.1in}
\displaystyle\frac{f_{j,x_i}}{3 f_j+f_i}\, \frac{f_{i,x_j}}{3f_i+f_j} =
\left(\displaystyle\frac{f_{i,x_j}}{3f_i+ f_j}\right)_{,x_i},  \quad i \neq j,
\end{array}
\right.
\label{eqtheorem41}
\end{equation}
\noindent where $h_i(x)=f_i\ \ e^{-2\int \frac{f_{i,x_j}}{3f_i+f_j} dx_j +  \psi(\hat{x_j})}$, the function $\psi(\hat{x_j})$ does not depend on 
$x_j$, satisfies the following system of ($n-1$) differential equations
\begin{equation}
\psi_{,x_i}= \displaystyle\int \left(\displaystyle\frac{f_{i,x_j}}{3f_i +f_j}\right)_{,x_i} dx_j - \displaystyle\frac{f_{j,x_i}}{3f_j +f_i}, \quad \mbox{for} \quad i\neq j
\label{eqcit41}
\end{equation}
\noindent and, up to a multiplicative constant,
\[
\varphi(x)= \exp\left(-\displaystyle\int \frac{f_{i,x_j}}{3f_i +f_j} dx_j + \psi(\hat{x_j})\right).
\]
\end{theorem}

\begin{proof}
 
 Suppose $\bar{g}=g/\varphi^2$  is a solution of  (\ref{eqtensorcurvature21}). Then, by the Lemma \ref{lemma41}, the first equation
of (\ref{eqtheorem41}) is satisfied and we obtain, for a fixed $j = 1, \ldots, n$, that
\[
\varphi(x) = \exp\left(-\displaystyle\int \frac{f_{i,x_j}}{3f_i +f_j}   dx_j + \psi(\hat{x_j})\right), \quad i \neq j,
\]
where the function $\psi(\hat{x_j})$ does not depend on $x_j$.

We will show that the expression of $\varphi$ is independent of the variable of integration and that the function $\psi$ is
well-defined.

Taking the derivative of $\varphi$ with respect to $ x_i$, $i \neq j$ we obtain
\[
\psi_{,x_i} = \displaystyle\int \left( \frac{f_{i,x_j}}{3f_i +f_j}\right )_{,x_i} dx_j +
\displaystyle\frac{\varphi_{,x_i}}{\varphi} = \displaystyle\int \left( \frac{f_{i,x_j}}{3f_i +f_j}\right )_{,x_i} dx_j
- \frac{f_{j,x_i}}{3f_j +f_i}.
\]
Now, taking the derivative of this expression with respect to $x_k$, $k \neq j$, we obtain
\[
\psi_{,x_ix_k} = \displaystyle\int \left( \frac{f_{i,x_j}}{3f_i +f_j}\right )_{,x_i x_k} dx_j
- \left(\frac{f_{j,x_i}}{3f_j +f_i}\right)_{,x_k}.
\]
Similarly, we obtain that
$ \psi_{,x_kx_i} = \displaystyle\int \left( \frac{f_{i,x_j}}{3f_i +f_j}\right )_{,x_k x_i} dx_j
- \left(\frac{f_{j,x_k}}{3f_j +f_k}\right)_{,x_i}$.

Thus, $\psi_{,x_ix_k} = \psi_{,x_kx_i}$ if and only if
\[
\left(\frac{f_{j,x_i}}{3f_j +f_i}\right)_{,x_k} =
 \left(\frac{f_{j,x_k}}{3f_j +f_k}\right)_{,x_i} .
\]
Therefore, the second equation in (\ref{eqtheorem41}) is satisfied.

Taking the derivative of $\psi_{,x_i}$ with respect to $x_j$, $i \neq j$, we obtain
\[
\psi_{,x_ix_j} =  \left( \frac{f_{i,x_j}}{3f_i +f_j}\right )_{,x_i} - \left(\frac{f_{j,x_i}}{3f_j +f_i}\right)_{,x_j}.
\]

Since $\psi_{,x_ix_j} = \psi_{,x_jx_i} = 0$, the third equation in (\ref{eqtheorem41}) is satisfied.

Using the equation (\ref{eqlemma41}) we have that, for a fixed $j = 1, \ldots, n$,

\[
\ln \varphi (x) = - \displaystyle\int \displaystyle\frac{f_{i,x_j}}{3f_i+f_j}dx_j + \psi (\hat{x_j}), \quad i\neq j,
\]
\noindent where the function $\psi (\hat{x_j})$ does not depend on $x_j$.

Integrating the equation (\ref{eqlemma41}) with respect to another variable $x_s$, for a fixed $s = 1, \ldots, n$, $s \neq j$, we obtain
\[
\ln \tilde{\varphi} (x) = - \displaystyle\int \displaystyle\frac{f_{i,x_s}}{3f_i+f_s}dx_s + \tilde{\psi} (\hat{x_s}), \quad i\neq s,
\]
\noindent where the function $\tilde{\psi} (\hat{x_s})$ does not depend on $x_s$.

Setting $A = \ln \varphi (x) - \ln \tilde{\varphi} (x)$ we will show that $A$ is constant.

Taking the derivative with respect to $x_s$ we obtain

\[
A_{,x_s}  =  - \displaystyle\int \left ( \frac{f_{i,x_j}}{3f_i+f_j}\right )_{,x_s} dx_j +
\psi_{,x_s} + \frac{f_{i,x_s}}{3f_i+f_s}.
\]

Using the equation (\ref{eqlemma411}) and that $\psi_{,x_s} = \displaystyle\int \left( \frac{f_{i,x_j}}{3f_i+f_j}\right)_{,x_s} dx_j
 - \displaystyle\frac{f_{i,x_s}}{3f_i+f_s}$ we conclude that $A_{,x_s} =0$. Similarly, taking the derivative of $A$ with respect to $x_j$, 
 using the equation (\ref{eqlemma411}) and that expression of $\tilde\psi_{,x_j}$ we conclude that $A_{,x_j} =0$.

Now, for $k \neq s$ and $k \neq j$, using the expressions of the derivatives of $\psi$ and $\tilde \psi$ and the 
equation (\ref{eqlemma411}) from Lemma \ref{lemma41} we obtain
\[
    \begin{array}{lcl}
    A_{,x_k} & = &  - \displaystyle\int \left ( \frac{f_{i,x_j}}{3f_i+f_j}\right )_{,x_k} dx_j + \psi_{,x_k}
+  \displaystyle\int \left ( \frac{f_{i,x_s}}{3f_i+f_s}\right )_{,x_k} dx_s - {\tilde \psi}_{,x_k} \\
     & = & - \displaystyle\int \left ( \frac{f_{i,x_j}}{3f_i+f_j}\right )_{,x_k} dx_j +
\displaystyle\int \left ( \frac{f_{i,x_j}}{3f_i+f_j}\right )_{,x_k} dx_j + \frac{\varphi_{,x_k}}{\varphi} \\
& & + \displaystyle\int \left ( \frac{f_{i,x_s}}{3f_i+f_s}\right )_{,x_k} dx_s
  - \displaystyle\int \left ( \frac{f_{i,x_s}}{3f_i+f_s}\right )_{,x_k} dx_s -
\displaystyle\frac{{\tilde\varphi}_{,x_k}}{\tilde\varphi}\\
& = & - \displaystyle\frac{f_{i,x_k}}{3f_i+f_k} + \displaystyle\frac{f_{i,x_k}}{3f_i+f_k} = 0.
\end{array}
\]
Thus, $\varphi$ is well-defined. Since $\varphi$ is a solution of (\ref{eqtensorcurvature21}) then $\varphi$ satisfies (\ref{eqtensorcurvature3}).
From Lemma \ref{lemma41},  for $i \neq j $, we have
$ \varphi_{,x_i} = - \varphi \displaystyle\frac{f_{j,x_i}}{3f_j+f_i}$.
Taking the derivative with respect to $x_i$, we obtain
\[
\displaystyle\frac{\varphi_{,x_ix_i}}{\varphi} = -\displaystyle\frac{\varphi_{,x_i}}{\varphi}
\displaystyle\frac{f_{j,x_i}}{3f_j+f_i}
 - \left(\displaystyle\frac{f_{j,x_i}}{3f_j+f_i}\right)_{,x_i} = \left(\displaystyle\frac{f_{j,x_i}}{3f_j+f_i}\right)^2 -
\left(\displaystyle\frac{f_{j,x_i}}{3f_j+f_i}\right)_{,x_i}
\]
Thus, using that $ |\nabla_g \varphi|^2 = \displaystyle\sum_{k=1}^{n} (\varphi_{,x_k})^2 =
\displaystyle\sum_{k\neq i}\varphi^2 \left(\displaystyle\frac{f_{j,x_k}}{3f_j+f_k}\right)^2 +
\left(\varphi_{,x_i}\right)^2$ and the expression above, the first equation in (\ref{eqtensorcurvature3})
 is equivalent to
\[
\left(\displaystyle\frac{f_{j,x_i}}{3f_j+f_i}\right)^2 - \left(\displaystyle\frac{f_{j,x_i}}{3f_j+f_i}\right)_{,x_i} -
\displaystyle\frac{1}{2} \displaystyle\sum_{k\neq i}\left(\displaystyle\frac{f_{j,x_k}}{3f_j+f_k}\right)^2 -
\displaystyle\frac{1}{2} \left(\displaystyle\frac{f_{j,x_i}}{3f_j+f_i}\right)^2 = \varphi^2 f_i .
\]
Simplifying this expression, we obtain
\[
 \displaystyle\frac{1}{2}
\left(\frac{f_{j,x_i}}{3f_j +f_i}\right)^2 -
 \left(\displaystyle\frac{f_{j,x_i}}{3f_j +f_i}\right)_{,x_i}
- \displaystyle \sum_{k \neq i}\left(\displaystyle\frac{f_{j,x_k}}{3f_j +f_k}\right)^2 = f_i 
e^{(-2\displaystyle\int \frac{f_i}{3f_i+f_j}  dx_j +  \psi(\hat{x_j}))}.
\]
which proves the fourth equality of (\ref{eqtheorem41}).

From Lemma \ref{lemma41}, we obtain for $i \neq j$ that
\[
\begin{array}{lcl}
 \varphi_{,x_jx_i} & = & - \varphi_{,x_i} \displaystyle\frac{f_{i,x_j}}{3f_i+f_j} -
\varphi\left(\displaystyle\frac{f_{i,x_j}}{3f_i+f_j}\right)_{, x_i}\\
& = & \varphi \left\{ \displaystyle\frac{f_{j,x_i}}{3f_j+f_i}
\displaystyle\frac{f_{i,x_j}}{3f_i+f_j} - \left(\displaystyle\frac{f_{i,x_j}}{3f_i+f_j}\right)_{,x_i} \right\} = 0.
\end{array}
\]
Since $\varphi \neq 0$ the other expression is equal zero,  which proves the fifth and last equality of (\ref{eqtheorem41}).
The converse is a straightforward computation.
\end{proof}

In order to provide explicit examples of metrics satisfying ${\bar R} = R$, we shall consider particular cases for $T$.

\begin{theorem}\label{theorem42}
 Let $(\mathbb{R}^n, g)$, $n\geq 3$, be the euclidean space, with coordinates
  $x_1,...,x_n$, and metric $g_{ij}=\delta_{ij}$.
Then there exists a metric
  $\bar{g} = \displaystyle\frac{1}{\varphi^2} g$
such that  ${\bar{R}} = R = f (g\odot g)$, where $f$ is a nonvanishing smooth function, if and only if
\begin{equation}
f(x) = \displaystyle\frac{-\lambda }{2 (\displaystyle\sum_{i=1}^{n}( a x_i^2 + b_i x_i) + c)^4 }
\label{eqTfg41}
\end{equation}
\noindent where $a,b_i , c$ are constants, $ \lambda = \displaystyle\sum_{i=1}^n b_i^2 - 4ac$ and
\begin{equation}
\varphi (x) = \displaystyle\sum_{i=1}^{n}( a x_i^2 + b_i x_i) + c.
\label{eqTfg42}
\end{equation}
\noindent Any such metric  $\bar{g}$  is unique up to homothety.
Moreover, we have:
\begin{enumerate}
\item  if $\lambda < 0$ then $\bar g$ is globally defined on $\mathbb{R}^n$;
\item  if $\lambda \geq 0$ then the set of singularities of $\bar g$ is
\begin{enumerate}
 \item the empty set if $\lambda =0$ and $a=0$;
 \item a point if $\lambda =0$ and $a \neq 0$;
 \item a hyperplane if $\lambda>0$ and $a=0$;
 \item an $(n-1)$-dimensional  sphere if $\lambda>0$ and $a\neq 0$.
\end{enumerate}
\end{enumerate}
\end{theorem}

\begin{proof}
 Since $f_i=f_j$, for all $i,j$ from Lemma \ref{lemma41} we have
\begin{equation}
 \displaystyle\frac{\varphi_{,x_j}}{\varphi} = -\displaystyle\frac{f_{,x_j}}{4\,f} \quad \mbox{for all} \quad j.
\label{eqtfg4}
\end{equation}

Therefore, there exist a constant $\lambda$ such that $\varphi^4 f=-\displaystyle\frac{\lambda}{2}$ and soon
\linebreak $f = -\displaystyle\frac{\lambda}{2 \varphi^4}$.

From (\ref{eqtensorcurvature3}) we have that
\[
\varphi_{,x_ix_i} = f \varphi^3 + \frac{|\nabla_g\varphi|^2}{2\varphi}, \quad \mbox{for all} \quad  i = 1, ...,n.
\]
Then
\[
\varphi_{,x_ix_i} = \varphi_{,x_jx_j}
\]
\noindent for all $i,j$. Thus, for every $i = 1,...,n$,
$\varphi_i (x_i) = ax_i^2+b_i x_i + c_i$ and
\[\varphi(x) = \displaystyle\sum_{i=1}^n  \varphi_i(x_i)=
 \displaystyle\sum_{i=1}^n(ax_i^2+b_ix_i) + c.
\]
Using the above relation between $\varphi$ and $f$, we obtain that
\[
f(x) = \displaystyle\frac{-\lambda }{2 (\displaystyle\sum_{i=1}^{n}(a x_i^2 + b_i x_i) + c)^4 } .
\]

Calculating the expressions in the first equality in (\ref{eqtensorcurvature3}) we obtain that \linebreak
 $\lambda = \sum_i b_i^2 - 4ac$. Analyzing the expression of $\varphi$ we arrive to the conclusions concerning 
the domain of $\varphi$.
Particularly, if $\lambda < 0$ the function $\varphi$ does not vanish and the metric
$\bar g$ is globally defined on $\mathbb{R}^n$.
\end{proof}

\begin{remark} In this theorem, although we proved directly from the system (\ref{eqtensorcurvature3}),
the equations (\ref{eqtheorem41}) of the Theorem \ref{theorem41} are satisfied with $f_i=f$ $\forall i=1,...,n$ and $\varphi$
given by (\ref{eqTfg42}) is the same as in the Theorem \ref{theorem41} with $\psi(\hat{x_j})$ being a constant. This demonstrates that there exist examples of tensor in $\mathbb{R}^n$ that are solutions of the equations of Theorem \ref{theorem41}.
\end{remark}

Now, we present a result of non-existense of conformal metrics to a special kind of tensor $R$.

\begin{theorem}{\label{theorem43}}
  Let $(\mathbb{R}^n, g)$, $n \geq 3$, be the euclidean space, with coordinates
  $x_1,...,x_n$, and metric $g_{ij}=\delta_{ij}$.
Consider the (0,4)-tensor $R=T\odot g$, where $T = \displaystyle\sum_{i=1}^n f_i(x_i) dx_i^2$, with $f_i(x_i)$
 smooth functions that depend on only the variable $x_i$ such that $3f_i(x_i) + f_j(x_j) \neq 0$ for all $ i \neq j$.
Then there is no metric $\bar{g} = \displaystyle\frac{1}{\varphi^2} g$ such
that ${\bar R} = R$.
\end{theorem}

\begin{proof}
 Since $f_i$ do not depend on the variables $x_j, \ j \neq i$, from (\ref{eqlemma41}) we have
\begin{equation}
 \displaystyle\frac{\varphi_{,x_j}}{\varphi} = -\displaystyle\frac{f_{i,x_j}}{3f_i+f_j} = 0
\quad \mbox{for all} \quad j \neq i.
\end{equation}
and
\begin{equation}
 \displaystyle\frac{\varphi_{,x_i}}{\varphi} = -\displaystyle\frac{f_{j,x_i}}{3f_j+f_i} = 0
\quad \mbox{for all} \quad  i\neq j.
\end{equation}

Thus, $\varphi_{,x_k} = 0$ for all $k$, and therefore $\varphi$ is constant. Using (\ref{eqtensorcurvature3}) we conclude that $f_i=0$, for all $i=1,\ldots, n$, which contradicts $3f_i+f_j\neq 0$, for $i\neq j$ and therefore does not exist a metric $\bar g$ such that ${\bar R} = R$.
\end{proof}

In the particular case in which the components of the tensor $T$ depend only one variable,
we have the following result.

\begin{theorem}{\label{theorem44}}
 Let $(\mathbb{R}^n, g)$, $n \geq 3$, be the euclidean space, with coordinates
  $x_1,...,x_n$, and metric $g_{ij}=\delta_{ij}$. Consider a (0,4)-tensor $R=T\odot g$ where $T$ is a diagonal (0,2)-tensor given by 
$T = \displaystyle\sum_{i=1}^n f_i(x_k) dx_i^2$, with $f_i(x_k)$ smooth functions that depend only on $x_k$,
for some fixed $k, \ 1 \leq k \leq n$ such that $3f_i(x_k) + f_j(x_k) \neq 0$ for all $ i \neq j$.
There exists a metric   $\bar{g} = \displaystyle\frac{1}{\varphi^2} g$
such that  ${\bar{R}} = R$ if and only if all the functions $f_i$ for $i \neq k$ are equal to one another, say $f_i = f$ for $i \neq k$, and the functions
$f$ and $f_k$ satisfy the system
\begin{equation}
\left\{ \begin{array}{l}
\displaystyle\frac{1}{2} \left( \frac{f_{,x_k}}{3f +f_k}\right)^2 -
\left(\displaystyle\frac{f_{,x_k}}{3f +f_k}\right)_{,x_k} =
C^2 f_k \, v  \\
- \left(\displaystyle\frac{f_{,x_k}}{3f +f_k}\right)^2 = 2C^2 f \, v
\end{array}
\right.
\label{eqt1v41}
\end{equation}
\noindent where $v = v(x_k) = e^{-2 \int \frac{f_{,x_k}}{3f +f_k}  dx_k}$ and if $f_k$ and $f$ satisfy these conditions, then $\varphi$ depends only on $x_k$ and is given by
\begin{equation}
 \varphi(x_k) = C \exp (- \int \frac{f_{,x_k}}{3f +f_k} dx_k ),
\label{eqt1v42}
\end{equation}
\noindent where $C$ is a positive constant.
\end{theorem}

\begin{proof}
 Since $f_i = f_i(x_k)$ for some fixed $k$, from the Lemma \ref{lemma41} we have that $\varphi_{,x_j} = 0$
for every $j \neq k$ and therefore $\varphi = \varphi(x_k)$.
Furthermore, from the equality in (\ref{eqtensorcurvature3}) we have that
$f_i = - \displaystyle\frac{|\nabla_g\varphi|^2}{2\varphi^4}$ for all $ i \neq k$ and,
therefore $f_i = f_j$, if $i \neq k$ and $j\neq k$. From this, the expression of $\varphi$ in
(\ref{eqt1v42}) and the system in (\ref{eqt1v41}) are consequences of the Theorem \ref{theorem41}.
\end{proof}

\begin{corollary}{\label{corollary45}}
 Let $(\mathbb{R}^n, g)$, $n \geq 3$, be the euclidean space, with coordinates
  $x_1,...,x_n$, and metric $g_{ij}=\delta_{ij}$. Consider the (0,2)-tensor
\[
T =  f_k(x_k) dx_k^2 + f(x_k) \displaystyle\sum_{i\neq k} dx_i^2,
\]
\noindent where $f_k(x_k)=\displaystyle\frac{h^2 -2h_{,x_k}}{2C^2} e^{2\int h(x_k)  dx_k}$, 
$f(x_k)= - \displaystyle\frac{h^2}{2c^2}e^{2\int h(x_k)  dx_k}$, with $h=h(x_k)$  a smooth function that depends only on $x_k$, for some fixed $k, \ 1 \leq k \leq n$.
Then there exist a conformal metric $\bar g = g/\varphi^2$ such that ${\bar R} = R$ and
\begin{equation}
 \varphi(x_k)= C \exp(- \int h(x_k) \ dx_k),
\label{eqt1v43}
\end{equation}
\noindent where $C$ is a positive constant.

If, in addition, $0 \leq \left \vert \displaystyle\int h(x_k) \ dx_k \right \vert \leq L$, for a finite constant $L$, then the
metric $\bar g$ is complete on $\mathbb{R}^n$.
\end{corollary}

\begin{proof}
 Follows immediately from Theorem \ref{theorem44} considering $h(x_k) = \displaystyle\frac{ f_{,x_k}}{3f + f_k}$. The
equalities in (\ref{eqt1v41}) are trivially satisfied and the expressions of $f$ and $f_k$ are exactly the components
of the tensor $T$.
\end{proof}

We can extend the Theorem \ref{theorem41} to locally conformally flat manifolds.
Consider $(M^n, g)$, a Riemannian manifold locally conformally flat. We may consider the problem (\ref{eqtensorcurvature1}) for neighborhood  $V \subset M $, with local coordinates $(x_1,x_2, ... , x_n)$ such that $g_{ij} = \delta_{ij}/F^2 $, where $F$ is a nonvanishing smooth function on $V$.

\begin{theorem}{\label{theorem46}}
 Let $(M^n, g)$, $n\geq 3$, be  Riemannian manifold, locally conformally flat. Let $V$ be an open subset of $M$
 with coordinates  $x=(x_1,x_2, ... , x_n)$ with  $g_{ij} = \delta_{ij}/F^2 $. Consider a (0,4)-tensor $R=T\odot g$, where $T$ is a diagonal (0,2)-tensor given by $T = \displaystyle\sum_{i=1}^nf_{i}(x)dx_i^{2} $, with $f_{i}$ are smooth functions such that 
$3f_i(x) + f_j(x) \neq 0$ for all $x \in V$ and all $ i \neq j$.
Then there exists a metric   $\bar{g} = \displaystyle\frac{1}{\phi^2} g$ such that  ${\bar{R}} = R$ if and only if
the functions $f_i$, $\varphi$ and $\psi$ are given as in Theorem \ref{theorem41} and $\phi = \displaystyle\frac{\varphi}{F}$.
\end{theorem}

\begin{proof}
 We consider $\varphi = \phi F$ and apply Theorem \ref{theorem41}.
\end{proof}

\begin{remark}
In a similar fashion, we can extend the Theorem \ref{theorem44} for locally
conformally flat manifolds.
\end{remark}

As an application of the Theorem \ref{theorem46} we show that given a certain (0,4)-tensor $R$ in $\mathbb{R}_+^n$ there exist a metric $\bar{g}$, conformal to the metric of the hyperbolic space whose Riemannian curvature tensor is $R$.

\begin{example} Let $\mathbb{H}^n = (\mathbb{R}_+^n, g)$ the hyperbolic space, where $g = \displaystyle\frac{1}{x_n^2} g_0$, with 
$(g_0)_{ij} = \delta_{ij}$ the euclidean metric and $\mathbb{R}_+^n = \{x=(x_1,x_2, ... , x_n) \in \mathbb{R}^n / x_n > 0\}$.

Given the (0,4)-tensor $R=T\odot g$, where $T$ is the diagonal (0,2)-tensor
\[
 T = - \displaystyle\frac{(2x_n^2 - 1)^2}{2x_n^4} e^{2x_n^2} \sum_{i\neq n} dx_i^2 +
\displaystyle\frac{4x_n^4 - 8 x_n^2 -1}{2x_n^4} e^{2x_n^2} dx_n^2
\]
defined in $\mathbb{R}_+^n$, Theorem \ref{theorem46} garantees the existence of a metric $\bar{g} = \displaystyle\frac{1}{\phi^2} g$, where $ \phi(x) = \displaystyle\frac{\varphi(x)}{F(x)} =  e^{-x_n^2}$ such that $\bar{R} = R$.

Moreover, since $\phi(x)$ is bounded, $(\mathbb{R}_+^n, \bar{g})$ is a complete Riemannian manifold, conformal to the hyperbolic space.

The scalar curvature of $(\mathbb{R}_+^n, \bar{g})$ is not constant and given by
\[
\bar{K} = (n-1)e^{-2 x_n^2}(4(2-n)x_n^4 + 4(n-3) x_n^2 -n)
\]
\noindent and the Ricci tensor of the metric $\bar{g}$ is
\[
Ric_{\bar{g}} = \displaystyle\frac{4(2-n)x_n^4 + 2(2n-5) x_n^2 +1 -n}{x_n^2} \sum_{i \neq n}dx_ i^2
 + (n-1)\displaystyle\frac{4x_n^4 - 4 x_n^2 - 1}{x_n^2} dx_n^2.
\]

Likewise, $(\mathbb{R}_+^n, \bar{g})$ does not have constant sectional curvature, given by
\[
K\left(\displaystyle\frac{\partial}{\partial x_i},\displaystyle\frac{\partial}{\partial x_j}\right)
 =-(1-2x_n^2)^2 e^{-2 x_n^2} \leq 0
\]
\noindent if $i,j \neq n$ and
\[
K\left(\displaystyle\frac{\partial}{\partial x_i},\displaystyle\frac{\partial}{\partial x_n}\right)
=2x_n^2(2x_n^2 -3)   e^{-2 x_n^2} .
\]
\end{example}

\begin{example} Corollary \ref{corollary45} also provide examples of
complete, conformally flat manifolds with prescribed Riemannian curvature tensor and non-constant curvatures.

\begin{enumerate}
\item In the euclidean space $(\mathbb{R}^n, g)$, $n \geq 3$, consider the (0,4)-tensor $R=T\odot g$, where $T$ is a diagonal (0,2)-tensor
given by
\[
T= \left(\displaystyle\frac{\sinh^2 x_k - 2 \cosh x_k}{2C^2} e^{2\cosh x_k}\right) dx_k^2 -
 \displaystyle\frac{\sinh^2 x_k}{2C^2}e^{2\cosh x_k} \displaystyle\sum_{i \neq k} dx_i^2,
\]
\noindent where $C$ is a positive constant. 

Corollary \ref{corollary45} garantees the existence of a metric
$\bar{g} = \displaystyle\frac{1}{\varphi^2} g$, conformal to the euclidean metric, such that $\bar{R} = R = T \odot g$ is the Riemannian curvature tensor of the
metric $\bar{g}$. In particular, we have that
\[
\varphi(x_k) = C e^{-\cosh x_k}
\]
\noindent where $C$ is a positive constant. The manifold $(\mathbb{R}^n, \bar{g})$ is complete,
has negative scalar curvature given by
\[
\bar{K} = -(n-1) C e^{-\cosh x_k}(2\cosh^2 x_k +(n-2)\sinh^2 x_k)
\]
\noindent and negative Ricci curvature whose Ricci tensor is negative definite and given by
\[
Ric_{\bar{g}} = -(n-1)\cosh x_k dx_k^2 - (\cosh x_k + (n-2)\sinh^2 x_k)\displaystyle\sum_{i \neq k}dx_ i^2.
\]
Moreover,  $(\mathbb{R}^n, \bar{g})$ has nonpositive sectional curvature given by
\[
K\left(\displaystyle\frac{\partial}{\partial x_i},\displaystyle\frac{\partial}{\partial x_j}\right)
 =-C^2 \sinh^2 x_ k e^{-2\cosh x_k}
\]
\noindent if $i,j \neq k$ and
\[
K\left(\displaystyle\frac{\partial}{\partial x_i},\displaystyle\frac{\partial}{\partial x_k}\right)
=-C^2 \cosh x_k e^{-2 \cosh x_k}.
\]

\item In the euclidean space $(\mathbb{R}^n, g)$, $n \geq 3$, consider the (0,4)-tensor $R=T\odot g$, where $T$ is a diagonal (0,2)-tensor given by
\[
T= \displaystyle\frac{(4 x_k^2 - 2) }{C^2} dx_k^2 -
 \displaystyle\frac{2 x_k^2}{C^2} \displaystyle\sum_{i \neq k} dx_i^2,
\]
\noindent where $C$ is a positive constant. Corollary \ref{corollary45} garantees the existence of a metric
$\bar{g} = \displaystyle\frac{1}{\varphi^2} g$, conformal to the euclidean metric, such that $\bar{R} = R = T \odot g$ is the Riemannian curvature tensor of the
metric $\bar{g}$. In particular, we have that
\[
\varphi(x_k) = \displaystyle\frac{C}{1+ x_k^2}
\]
\noindent where $C$is a positive constant. The manifold $(\mathbb{R}^n, \bar{g})$ is complete,
has negative scalar curvature given by
\[
\bar{K} = -\displaystyle\frac{4(n-1)C^2}{(1+x_k^2)^2}(1+ (n-3)x_k^2)
\]
\noindent and the Ricci tensor of $\bar{g}$ is given by
\[
Ric_{\bar{g}} = \displaystyle\frac{2(n-1)(x_k^2 -1)}{(1+x_k^2)^2} dx_k^2 +
\displaystyle\frac{(10-4n)x_k^2 -2}{(1+ x_k^2)^2} \displaystyle\sum_{i \neq k}dx_ i^2.
\]
Moreover,  $(\mathbb{R}^n, \bar{g})$ has sectional curvature given by
\[
K\left(\displaystyle\frac{\partial}{\partial x_i},\displaystyle\frac{\partial}{\partial x_j}\right)
=-\displaystyle\frac{4 C^2 x_k^2}{(1+x_k^2)^4}
\]
\noindent if $i,j \neq k$ and
\[
K\left(\displaystyle\frac{\partial}{\partial x_i},\displaystyle\frac{\partial}{\partial x_k}\right)
=-\displaystyle\frac{2 C^2 (1- x_k^2)}{(1+x_k^2)^4}.
\]
\item In the euclidean space $(\mathbb{R}^n, g)$ consider the (0,4)-tensor $R=T\odot g$, where $T$ is a diagonal (0,2)-tensor
given by
\[
T= \displaystyle\frac{2(x_k^2-1)}{C^2} e^{2x_k^2} dx_k^2 -
 \displaystyle\frac{2x_k^2}{C^2}e^{2x_k^2} \displaystyle\sum_{i \neq k} dx_i^2,
\]
\noindent where $C$ is a positive constant. Corollary \ref{corollary45} garantees the existence of a metric
$\bar{g} = \displaystyle\frac{1}{\varphi^2} g$, conformal to the euclidean metric, such that $\bar{R} = R = T \odot g$ is the Riemannian curvature tensor of the
metric $\bar{g}$. In particular, we have that
\[
\varphi(x_k) = C e^{-x_k^2}
\]
\noindent where $C$ is a positive constant. The manifold $(\mathbb{R}^n, \bar{g})$ is complete, has negative scalar curvature given by
\[
\bar{K} = -4(n-1) C^2 e^{-2x_k^2}(1+(n-2)x_k^2)
\]
\noindent and negative Ricci curvature, whose Ricci tensor is given by
\[
Ric_{\bar{g}} = -2(n-1)dx_k^2 -2 (1-2(n-2)x_k^2)\displaystyle\sum_{i \neq k}dx_ i^2.
\]
The sectional curvature of $(\mathbb{R}^n, \bar{g})$ is nonpositive and given by the expressions
\[
K\left(\displaystyle\frac{\partial}{\partial x_i},\displaystyle\frac{\partial}{\partial x_j}\right)
=-4x_k^2C^2 e^{-2x_k^2}
\]
\noindent if $i,j \neq k$ and
\[
K\left(\displaystyle\frac{\partial}{\partial x_i},\displaystyle\frac{\partial}{\partial x_k}\right)
=-2C^2 e^{-2x_k^2}.
\]
We observe that although there are points where the tensor $R$ is zero, there still exists a 
complete metric such that the curvature tensor of this metric is the prescribed tensor $R$.  
\end{enumerate}
\end{example}

\bibliographystyle{amsplain}

\end{document}